\documentclass[review,onefignum,onetabnum]{siamart171218}

\usepackage{lipsum}
\usepackage{mathrsfs}
\usepackage{amsfonts}
\usepackage{graphicx}
\usepackage{epstopdf}
\usepackage{algorithmic}
\usepackage{amssymb}
\usepackage{amsmath}

\ifpdf
  \DeclareGraphicsExtensions{.eps,.pdf,.png,.jpg}
\else
  \DeclareGraphicsExtensions{.eps}
\fi

\newsiamremark{hypothesis}{Hypothesis}
\crefname{hypothesis}{Hypothesis}{Hypotheses}
\newsiamthm{claim}{Claim}

\nolinenumbers

\makeatletter
\g@addto@macro\bfseries{\boldmath}
\makeatother

\headers{Uniform asymptotic expansions for Whittaker functions}{T. M. Dunster}

\title{Uniform asymptotic expansions for the Whittaker functions $M_{\kappa,\mu}(\MakeLowercase{z})$ and $W_{\kappa,\mu}(\MakeLowercase{z})$ with $\mu$ large}

\author{T. M. Dunster\thanks{Department of Mathematics and Statistics, San Diego State University, 5500 Campanile Drive, San Diego, CA 92182-7720, USA. 
  (\email{mdunster@sdsu.edu}, \url{https://tmdunster.sdsu.edu}).}}

\usepackage{amsopn}

\makeatletter
\newcommand*{\addFileDependency}[1]{% argument=file name and extension
  \typeout{(#1)}% latexmk will find this if $recorder=0 (however, in that case, it will ignore #1 if it is a .aux or .pdf file etc and it exists! if it doesn't exist, it will appear in the list of dependents regardless)
  \@addtofilelist{#1}% if you want it to appear in \listfiles, not really necessary and latexmk doesn't use this
  \IfFileExists{#1}{}{\typeout{No file #1.}}% latexmk will find this message if #1 doesn't exist (yet)
}
\makeatother

\nolinenumbers

\begin{document}

\maketitle

\begin{abstract}
  Uniform asymptotic expansions are derived for Whittaker's confluent hypergeometric functions $M_{\kappa,\mu}(z)$ and $W_{\kappa,\mu}(z)$, as well as the numerically satisfactory companion function $W_{-\kappa,\mu}(ze^{-\pi i})$. The expansions are uniformly valid for $\mu \rightarrow \infty$, $0 \leq \kappa/\mu \leq 1-\delta <1$, and $0 \leq \arg(z) \leq \pi$. By using appropriate connection and analytic continuation formulas these expansions can be extended to all unbounded nonzero complex $z$. The approximations come from recent asymptotic expansions involving  elementary functions and Airy functions, and explicit error bounds are either provided or available.
\end{abstract}



\section{Introduction} 
\label{sec1}

The purpose of this paper is to obtain uniform asymptotic expansions as $\mu \rightarrow \infty$ and $0 \leq \kappa/\mu \leq 1-\delta$ for the Whittaker confluent hypergeometric functions $M_{\kappa,\mu}(z)$ and $W_{\kappa,\mu}(z)$; here and throughout $\delta$ is an arbitrary small positive constant.

Amongst their physical applications, these functions appear in solutions of the wave equation in paraboloidal coordinates \cite[Chap. 7]{Hochstadt:1971:FMP}, and also are used to describe the behavior of charged particles in a Coulomb potential \cite{Gaspard:2018:CFB}. In \cite{Dunster:1994:UAS} and \cite{Nestor:1984:UAA} they play a central role in the uniform asymptotic theory of differential equations having a coalescing turning point and simple pole.

$M_{\kappa,\mu}(z)$ and $W_{\kappa,\mu}(z)$ are solutions of the Whittaker differential equation
\begin{equation}  \label{02}
\frac{d^{2}y}{dz^{2}}=
\left(\frac{\mu^{2}-\frac{1}{4}}{z^{2}}-\frac{\kappa}{z}
+\frac{1}{4}\right)y,
\end{equation}
which has a regular singularity at $z=0$ and an irregular singularity at $z=\infty$. These two Whittaker functions can be expressed in terms of confluent hypergeometric functions via \cite[Eqs. 13.2.3, 13.4.4, 13.14.2 and 13.14.3]{NIST:DLMF}
%%%%%%%
\begin{equation}
\label{M}
M_{\kappa,\mu}(z)
=z^{\mu+\frac{1}{2}}e^{-z/2}
M\left(\tfrac{1}{2}+\mu-\kappa,1+2\mu,z\right)
\quad   (2\mu \neq -1,-2,-3,\cdots),
\end{equation}
%%%%%%%
and
%%%%%%%
\begin{equation}
\label{W}
W_{\kappa,\mu}(z)
=z^{\mu+\frac{1}{2}}e^{-z/2}
U\left(\tfrac{1}{2}+\mu-\kappa,1+2\mu,z\right),
\end{equation}
%%%%%%%
where
%%%%%%%
\begin{equation}
\label{MM}
M(a,b,z)
=\sum_{s=0}^{\infty}\frac{(a)_{s}}{(b)_{s}s!}z^{s},
\end{equation}
%%%%%%%
and
%%%%%%%
\begin{equation}
\label{U}
U(a,b,z)=\frac{1}
{\Gamma\left(a\right)}\int_{0}^{\infty}e^{-zt}
t^{a-1}(1+t)^{b-a-1}dt
\quad (\Re(a)>0,\,|\arg(z)| < \tfrac{1}{2}\pi).
\end{equation}
%%%%%%%

Assuming $\Re(\mu) \geq -\frac{1}{2}$ they are fundamental solutions that are recessive (bounded) at one of the singularities, according to the limiting behaviour \cite[Eqs. 13.13.14 and 13.14.21]{NIST:DLMF}
%%%%
\begin{equation}
\label{104}
M_{\kappa,\mu}(z)=z^{\mu +\frac{1}{2}}
\left\{1+\mathcal{O}(z)\right\}
\quad (z \rightarrow 0),
\end{equation}
%%%%%
and
%%%%
\begin{equation}
\label{110}
W_{\kappa,\mu}(z)=z^{\kappa}e^{- z/2}
\left\{1+\mathcal{O}\left(z^{-1}\right)\right\}
\quad (z \rightarrow \infty, \, |\arg(z)| 
\leq \tfrac{3}{2}\pi - \delta).
\end{equation}
%%%%%
All other independent solutions of (\ref{02}) are unbounded at these singularities. Generally both functions are multivalued, and with a branch cut taken along $-\infty < z \leq 0$ the principal branches for $M_{\kappa,\mu}(z)$ and $W_{\kappa,\mu}(z)$ correspond to those on the RHS of (\ref{104}) and (\ref{110}) respectively.

The differential equation (\ref{02}) is unchanged if $\kappa$ and $z$ both change sign. Thus a third solution is given by $W_{-\kappa,\mu}(ze^{-\pi i})$, and this is recessive at $z = \infty \exp(\pi i)$. As such, the three functions form a numerically satisfactory set of solutions in the upper half plane $0 \leq \arg(z) \leq \pi$ ($z \neq 0$), which we denote by $\mathbb{C}^{+}$. 

We shall obtain expansions for these three solutions, and we only need to consider $\mathbb{C}^{+}$, since we can extend our results to the lower half plane $-\pi \leq \arg(z) \leq 0$ ($0 \leq \arg(\bar{z}) \leq \pi$) via the Schwarz reflection formulas (assuming $\mu, \kappa \in \mathbb{R}$)
%%%%
\begin{equation}
\label{92a}
M_{\kappa,\mu}(z)
=\overline{M_{\kappa,\mu}(\bar{z})}, \quad
W_{\kappa,\mu}(z)
=\overline{W_{\kappa,\mu}(\bar{z})},
\end{equation}
%%%%%
as well as 
%%%%
\begin{equation}
\label{92b}
W_{-\kappa,\mu}\left(ze^{\pi i}\right)
=\overline{W_{-\kappa,\mu}\left(\bar{z}e^{-\pi i}\right)},
\end{equation}
%%%%%
which is a numerically satisfactory companion solution for $-\pi \leq \arg(z) \leq 0$. Furthermore, extensions of our results to other values of $\arg(z)$, i.e. across the cut $-\infty < z \leq 0$, come from using well-known analytic continuation formulas \cite[Sect. 13.14(ii)]{NIST:DLMF}. Included in these is the relation
%%%%
\begin{equation}
\label{92c}
M_{\kappa,\mu}\left(ze^{\pm\pi\mathrm{i}}\right)
=\pm i e^{\pm\mu\pi i}M_{-\kappa,\mu}(z),
\end{equation}
%%%%%
and so in conjunction with the new expansions for $W_{-\kappa,\mu}\left(\mu ze^{\pm \pi i}\right)$ our results also provide asymptotic expansions for all solutions when $0 \leq -\kappa/\mu \leq 1-\delta$ ($\mu > 0$).

The case we consider here, $\mu \rightarrow \infty$ with $0 \leq \kappa/\mu \leq 1-\delta <1$, was first studied by Olver \cite[Chap. 7, Sect. 11.1]{Olver:1997:ASF}. He obtained a one term approximation, complete with error bounds, for $M_{\kappa,\mu}(x)$ and $W_{\kappa,\mu}(x)$ with $x$ real and positive. The results here extend Olver's results in several aspects. Firstly, in \cref{sec2} we extend the approximations to expansions by using the recent results in \cite{Dunster:2020:LGE}. These are Liouville-Green (LG) expansions of a less used form, where the coefficients in the asymptotic expansions appear in the exponent of the approximating exponential function. This has the advantage over the standard LG expansions in that the coefficients involved can be evaluated simply and explicitly, as well as being accompanied by simpler and sharper error bounds (which we shall supply). 

Secondly, as we alluded to earlier, our expansions are valid for complex $z$. As we shall show, for large $\mu$ and $0 \leq \kappa/\mu \leq 1-\delta$ the equation has two conjugate turning points (as defined by \cite[Sect. 2.8(i)]{NIST:DLMF}) in the right half plane $\Re(z) \geq 0$. Our LG expansions, along with suitable connection formulas, are valid at all points in $\mathbb{C}^{+}$ except in the neighbourhood of the turning point in the first quadrant. 

In \cref{sec3} we go further by using recent results given in \cite{Dunster:2020:SEB} to obtain asymptotic expansions that are valid at this turning point, and in fact valid for all $z \in \mathbb{C}^{+}$. These involve Airy functions, and while the previous general theory of turning point expansions is well established (see \cite[Chap. 11]{Olver:1997:ASF} the coefficients are typically very difficult to compute beyond one term. A feature of the new asymptotic expansions in \cite{Dunster:2020:SEB} is that the coefficients that appear are significantly easier to compute, either directly if not too close to the turning point, or via Cauchy's integral theorem in a neighbourhood of the turning point. Moreover the explicit error bounds provided by \cite{Dunster:2020:SEB} are readily computable. We remark that the added complication of the turning point in the present application being complex-valued rather than lying on the real axis does not present a significant hurdle in the application of these recent Airy function expansions.

A number of powerful asymptotic approximations for Whittaker functions have previously been obtained by Olver. For large $\mu$ and fixed $\kappa$ he constructed a simple expansion for $M_{\kappa,\mu}(z)$ in \cite[Chap. 10, Ex. 3.4]{Olver:1997:ASF} which is valid for real or complex values of $z$ and the two parameters. Next, in his landmark paper \cite{Olver:1975:SOL} Olver obtained asymptotic approximations for solutions of differential equations having two coalescing turning points. In \cite{Olver:1980:WFW} this was applied to obtain approximations for the Whittaker functions in terms of parabolic cylinder functions, again for large $\mu$, that are valid for $-(1 - \delta)\mu \leq \kappa \leq \mu$ and also $\mu \leq \kappa \leq \mu/\delta$. The argument $z$ is either real or purely imaginary, as are the parameters $\kappa$ and $\mu$. Although the parameter range is large, the price is that these approximations consist of only one term rather than expansions, and involve complicated Liouville transformation variables and error bounds.

For large $|\mu|$ with $\mu$ or $\kappa$ imaginary see \cite[Chap. 11, Sect. 4.3]{Olver:1997:ASF} and \cite{Dunster:2003:UAW}. For large $\kappa$ and $0 \leq \mu \leq (1 - \delta)\kappa$ Dunster \cite{Dunster:1989:UAW} obtained expansions, with error bounds, for complex argument $z$ in terms of Bessel functions. This used an asymptotic theory of differential equations having a coalescing turning point and double pole \cite{Boyd:1986:UAS}. Also for $\kappa$ large we mention Olver's results \cite[Chap. 11, Ex. 7.3]{Olver:1997:ASF} for $W_{\kappa,\mu}(x)$, with $\mu$ real and bounded, and $x \in [\delta,\infty)$.

\section{Liouville-Green expansions} 
\label{sec2}

We begin by replacing $z$ by $\mu z$ in (\ref{02}) to obtain
%%%%
\begin{equation}
\label{88}
d^{2}w/dz^{2}=\left\{\mu^{2}f(z)+g(z)\right\}w,
\end{equation}
%%%%%
where for $\lambda=\kappa /\mu$
%%%%
\begin{equation}
\label{89}
f(z)=\frac{z^{2}-4\lambda z+4}{4z^{2}},
\quad g(z)=-\frac{1}{4z^{2}}.
\end{equation}
%%%%%
This differential equation has solutions $M_{\kappa,\mu}(\mu z)$, $W_{\kappa,\mu}(\mu z)$ and $W_{-\kappa,\mu}(\mu ze^{-\pi i})$. The equation (\ref{02}) has been recast into this form in order for the subsequent asymptotic approximations to be uniformly valid at both $z=0$ and at $z=\infty$: see \cite[Remark 1.2]{Dunster:2020:LGE} and \cite[Chap. 10, Thm. 4.1 and Ex. 4.1]{Olver:1997:ASF}.

For large $\mu$ it has turning points where $f(z)=0$, which for $0 \leq \kappa < \mu$ are found to be a pair of conjugate points in the right half plane located at $z=z^{\pm}(\lambda)$ say, given by
%%%%
\begin{equation}
\label{89a}
z^{\pm}(\lambda)=2\lambda \pm 2i\sqrt{1-\lambda^2}
=2 e^{\pm i \theta(\lambda)},
\end{equation}
%%%%%
where $\theta(\lambda) = \arccos(\lambda) \in (0,\frac{1}{2}\pi]$ for $0\leq \lambda < 1$.

Note that they coalesce into a double pole at $z=2$ when $\lambda=1$, and our expansions would no longer be valid on the positive real $z$ axis in this case. Moreover the expansions of \cref{sec3} break down when the two turning points are close or coincide. Thus, in order for them to be bounded away from one another we require in this section and the next that
%%%%
\begin{equation}
\label{90}
\lambda =\kappa /\mu \in \left[0,1-\delta \right].
\end{equation}
%%%%%
As we mentioned in \cref{sec1} the case $\lambda <0$ ($\kappa <0$) is covered by connection formulas.

Next, the following Liouville variable appears in all our expansions (see \cite[Eq. (1.3)]{Dunster:2020:LGE})
%%%%
\begin{equation}
\label{91}
\xi =\int f^{1/2}(z)dz=\frac{1}{2}Z
-\lambda \ln(Z+z-2\lambda)
-\ln \left(\frac{Z-\lambda z+2}{z}\right),
\end{equation}
%%%%%
where
%%%%
\begin{equation}
\label{92}
Z(z)=\left(z^{2}-4\lambda z+4\right)^{1/2}.
\end{equation}
%%%%%
We find as $z \rightarrow 0$ that $\Re(\xi) \rightarrow -\infty$ such that
%%%%
\begin{equation}
\label{103}
\xi =\ln \left(\tfrac{1}{4}z\right)+1-\lambda \ln(2-2\lambda)+\mathcal{O}(z),
\end{equation}
%%%%%
and $\Re(\xi) \rightarrow \pm \infty$ as $\Re(z) \rightarrow \pm \infty$ such that
%%%%
\begin{equation}
\label{109}
\xi =\tfrac{1}{2}z-\lambda\ln(2z)
-\lambda -\ln(1-\lambda)
+\mathcal{O}\left(z^{-1}\right).
\end{equation}
%%%%%
The branches of these multi-valued functions will be specified more precisely shortly. 

\begin{figure}[hthp]
  \centering
  \includegraphics[trim={0 210 0 210}, width=0.9\textwidth,keepaspectratio]{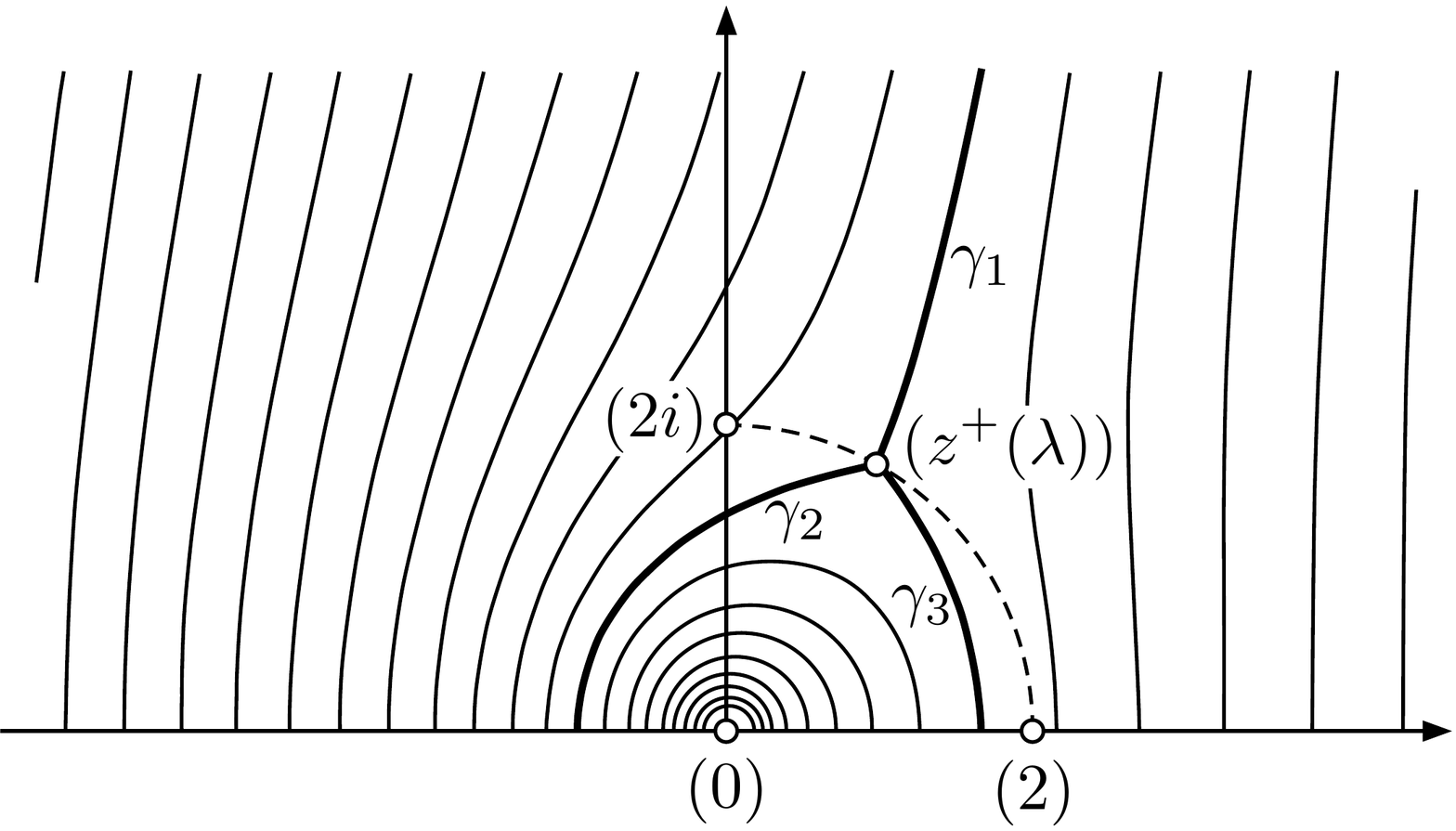}
  \caption{Level curves $\Re(\xi)=\mathrm{constant}$}
  \label{fig:fig1}
\end{figure}

%trim={0 100 0 100},

In \cref{fig:fig1} a number of level curves $\Re(\xi)=\mathrm{constant}$ are shown in the upper half plane under consideration, and these determine the regions of validity of our asymptotic expansions. Here we have used the particular value $\lambda=\frac{1}{2}$ ($z^{+}(\frac{1}{2})=1+i\sqrt{3}$), but the general configuration remains the same for $0\leq \lambda \leq 1-\delta$. The dashed curve is the quarter circle $|z|=2$, $0\leq \arg(z)\leq \frac{1}{2}\pi$, on which the turning point $z^{+}(\lambda)$ lies for $0\leq \lambda \leq 1$ (see (\ref{89a})).

The level curves labelled $\gamma_{j}$ ($j=1,2,3$) emanating from the turning point $z^{+}(\lambda)$ are particularly significant, and divide the half plane into three regions $S_{j}$ ($j=0,\pm 1$) as depicted in \cref{fig:fig2}. The subscripts are chosen to match the Airy functions used in \cref{sec3}.

The interior of the finite region $S_{-1}$ contains the singularity $z=z^{(1)}:=0$ and is where $M_{\kappa,\mu}(\mu z)$ is recessive (exponentially small for large $\mu$). The interior of $S_{0}$ containing $z=z^{(2)}:=+\infty$ is where $W_{\kappa,\mu}(\mu z)$ is recessive, and the interior of $S_{1}$ containing $z=z^{(3)}:=\infty e^{\pi i}$ is where $W_{-\kappa,\mu}(\mu ze^{-\pi i})$ is recessive. All three solutions are dominant (exponentially large for large $\mu$) in the exterior of their respective regions of recessiveness.

\begin{figure}[hthp]
  \centering
  \includegraphics[trim={0 170 0 180}, width=0.6\textwidth,keepaspectratio]{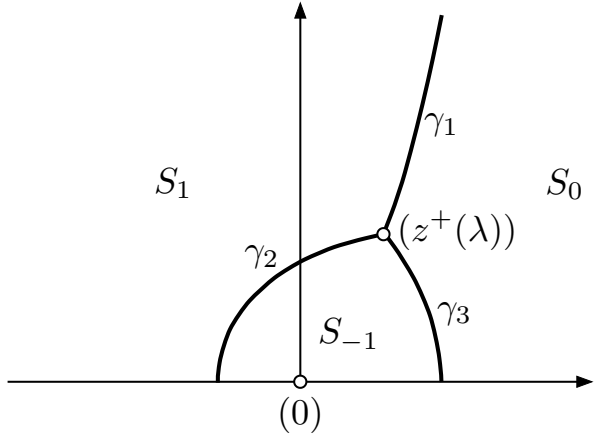}
  \caption{Regions $S_{j}$ ($j=0,\pm 1$)}
  \label{fig:fig2}
\end{figure}

Our LG expansions will be uniformly valid in unbounded domains $Z_{j}(\delta)$ ($j=1,2,3$) which consist of all points in $\mathbb{C}^{+}$ except those whose distance to the curve $\gamma_{j}$ is less than $\delta$. Note that this means the turning point $z^{+}(\lambda)$ does not lie in any of these three regions. The region $Z_{j}(\delta)$ has been defined to meet the requirement that all points $z$ in it can be linked to $z^{(j)}$ by a path $\mathscr{L}_{j}(z)$ (say) that (i) consists of a finite chain of $R_{2}$ arcs (as defined in \cite[Chap. 5, Sect. 3.3]{Olver:1997:ASF}), (ii) as $t$ passes along this path from $a _{j}$ to $z$, the real part of $\xi(t)$ is monotonic, and (iii) the path does not get closer than a distance $\delta$ to the turning point, i.e. $|t-z^{+}(\lambda)|\geq \delta$ for all $t \in \mathscr{L}_{j}(z)$. In condition (ii) $\xi(t)$ is given by (\ref{91}) with $z$ replaced by $t$ and with branches taken so that it is a continuous function of $t$ in $Z_{j}(\delta)$. We shall call these paths \emph{progressive}.

For $\xi = \xi(z)$ we take a branch cut along $\gamma_{2}$ and choose the branches in (\ref{91}) so that $\Re(\xi) \rightarrow +\infty$ as $\Re(z) \rightarrow z^{(2)}=+\infty$ and such that it is a continuous function (indeed analytic) in $\mathbb{C}^{+} \setminus \gamma_{2}$. Hence $\xi$ is real on the positive $z$ axis, and (from (\ref{103}) and (\ref{109})) recall that $\Re(\xi) \rightarrow -\infty$ as $\Re(z) \rightarrow z^{(1)}=0$ and also as $z \rightarrow z^{(3)}= \infty e^{\pi i}$. We shall call this the principal value of $\xi$ (as well as $Z$) in this cut half plane.

We shall allow $z$ to cross the branch cut as appropriate (and described in more detail below), and if so $\xi$ must vary continuously along a path that crosses this cut. Thus $\Re(\xi)$ will take the opposite sign if on a different sheet. For example, if $z$ travels along a ray from $0$ to $\infty$ in the left half plane (hence across the cut $\gamma_{2}$) then $\Re(\xi)$ will increase continuously and  monotonically from $-\infty$ to $+\infty$.

Coefficients that appear in the expansions are provided by \cite[Eqs. (1.10) - (1.12)]{Dunster:2020:LGE}. With the aid of (\ref{89}), (\ref{92}) and
%%%%
\begin{equation}
\label{95}
\frac{d\xi}{dz}=f^{1/2}(z)=\frac{Z}{2z},
\end{equation}
%%%%%
which comes from (\ref{91}), these are given by
%%%%
\begin{equation}
\label{97}
\hat{E}_{s}(z)=\frac{1}{2}\int_{\infty}^{z}
\frac{Z(t)\hat{F}_{s}(t)}{t} dt
\quad (s=1,2,3,\cdots),
\end{equation}
%%%%%
where
%%%%
\begin{equation}
\label{93}
\hat{F}_{1}(z)=
\frac{4f(z)f''(z) - 5f'^{2}(z)}
{32f^{3}(z)}+\frac{g(z)}{2f(z)}=
-\frac{z\left(z^{3}-16z+4\lambda^{2}z+16\lambda \right)}{2Z^{6}},
\end{equation}
%%%%%
%%%%
\begin{multline}
\label{94}
\hat{F}_{2}(z)=-\frac{z}{Z}\frac{d\hat{F}_{1}(z)}{dz}
=-\frac{z}{Z^{9}}
\left\{z^{5}+2\lambda z^{4}+8\left(\lambda^{2}-5\right)z^{3}-8\lambda \left(\lambda^{2}-9\right)z^{2} \right.
\\ \left.
-16\left(5\lambda^{2}-4\right)z-32\lambda \right\},
\end{multline}
%%%%%
and
%%%%
\begin{equation}
\label{96}
\hat{F}_{s+1}(z)=-\frac{z}{Z}\frac{d\hat{F}_{s}(z)}{dz}-\frac{1}{2}\sum_{j=1}^{s-1}\hat{F}_{j}(z)\hat{F}_{s-j}(z)
\quad (s=2,3,4,\cdots).
\end{equation}
%%%%%
The branch of each coefficient is the same as the principal one for $Z$, so that each $\hat{E}_{s}(z)$ continuous in $\mathbb{C}^{+}\setminus \gamma_{2}$, with $Z>0$ in (\ref{97}), (\ref{94}) and (\ref{96}) when $z \in (0,\infty)$.

Note that the arbitrary lower integration limit in (\ref{97}) was chosen for convenience so that
%%%%
\begin{equation}
\label{97a}
\lim_{z \rightarrow \infty}\hat{E}_{s}(z)=0
\quad (s=1,2,3,\cdots).
\end{equation}
%%%%%%%

From (\ref{97}) and (\ref{93}) we have for the first coefficient on explicit integration
%%%%
\begin{equation}
\label{98}
\hat{E}_{1}(z)=\frac{\lambda z^{3}+6\left(1-2\lambda^{2}\right)z^{2}+12\lambda^{3}z+8\lambda^{2}-16}{24\left(1-\lambda^{2}\right)Z^{3}}
-\frac{\lambda}{24\left(1-\lambda^{2}\right)}.
\end{equation}
%%%%%

Let us show how all the other coefficients can be explicitly evaluated. Firstly from (\ref{92})
%%%%
\begin{equation}
\label{99a}
\frac{dZ}{dz}=\frac{z-2\lambda}{Z},
\end{equation}
%%%%%
and so from (\ref{93}) - (\ref{96}) one can verify by induction that
%%%%
\begin{equation}
\label{99b}
\hat{F}_{s}(z)=z Z^{-3s-3} p_{2s+1}(z),
\end{equation}
%%%%%
where $p_{2s+1}(z)$ is a polynomial of degree $2s+1$. Now let
%%%%
\begin{equation}
\label{99c}
\tau=z-2\lambda,
\end{equation}
%%%%%
and
%%%%
\begin{equation}
\label{99cc}
\beta=\tau / Z.
\end{equation}
%%%%%
Then from (\ref{92}), (\ref{99c}) and (\ref{99cc})
%%%
\begin{equation}
\label{99ccc}
\tau=2\beta\left(
\frac{1-\lambda^{2}}{1-\beta^{2}}\right)^{1/2},
\end{equation}
%%%%%
and
%%%
\begin{equation}
\label{99d}
\frac{d \beta}{dz}
=\frac{4\left(1-\lambda^{2}\right)}{Z^{3}}.
\end{equation}
%%%%%

If we denote $Z(\tau+2\lambda)=\tilde{Z}(\beta)$ where $\tau = \tau(\beta)$ is given by (\ref{99ccc}), and so from (\ref{99cc})
%%%
\begin{equation}
\label{99e}
\tilde{Z}(\beta) = \frac{\tau}{\beta}
=2\left(\frac{1-\lambda^{2}}{1-\beta^{2}}\right)^{1/2}.
\end{equation}
%%%%%
Also, from (\ref{99c}) write
%%%
\begin{equation}
\label{99ee}
p_{2s+1}(z)=p_{2s+1}(\tau+2\lambda)=\tilde{p}_{2s+1}(\beta),
\end{equation}
%%%%%
and hence from (\ref{97}), (\ref{99b}) and (\ref{99d})
%%%
\begin{equation}
\label{99f}
\hat{E}_{s}(z)=\frac{1}{2}\int_{\infty}^{z}
\frac{p_{2s+1}(t)}{Z^{3s+2}(t)}dt
=\frac{1}{8\left(1-\lambda^{2}\right)}
\int_{1}^{\beta} \frac{\tilde{p}_{2s+1}(b)}
{\tilde{Z}^{3s-1}(b)}db.
\end{equation}
%%%%%

Next, on referring to (\ref{99e}), the integrand for the odd coefficients $\hat{E}_{2m+1}(z)$ ($m=0,1,2,\cdots$) can be expressed as
%%%
\begin{equation}
\label{99g}
\frac{\tilde{p}_{4m+3}(\beta)}
{\tilde{Z}^{6m+2}(\beta)}
=\left(\frac{\beta}{\tau} \right)^{6m+2}
\tilde{p}_{4m+3}(\beta) 
= \frac{1}{2^{6m+2}}
\left(\frac{1-\beta^{2}}
{1-\lambda^{2}}\right)^{3m+1}
\tilde{p}_{4m+3}(\beta).
\end{equation}
%%%%%
Now decompose $p_{4m+3}(\tau+2\lambda)$ into its even (e) and odd (o) parts as functions of $\tau$:
%%%
\begin{equation}
\label{99hh}
\tilde{p}_{4m+3}(\beta)=p_{4m+3}(\tau+2\lambda)
=p_{4m+3}^{(e)}(\tau)
+p_{4m+3}^{(o)}(\tau).
\end{equation}
Thus from (\ref{99ccc}) it is clear that the fractional powers of $\beta$ in (\ref{99g}) occur for the odd powers of $\tau$, which come only from $p_{4m+3}^{(o)}(\tau)$. All these terms are of the form
%%%
\begin{equation}
\label{99h}
\mathrm{constant}\times\beta^{2s+1}\left(1-\beta^{2}\right)^{3m-s+\frac{1}{2}}
\quad (s=0,1,2,\cdots 2m+1),
\end{equation}
%%%%%
where each constant involves $\lambda$ alone. These are easy to integrate, using the substitution $u=1-\beta^{2}$. The contributions from $p_{4m+3}^{(e)}(\tau)$ are polynomials in $\beta^{2}$ which of course are trivial to integrate.

For the even coefficients $\hat{E}_{2m}(z)$ ($m=1,2,3,\cdots$) we can avoid integration completely, since they can be expressed explicitly in terms of $\hat{F}_{2m+1}(z)$ ($m=0,1,2,\cdots$). Specifically, from \cite[Eq. (1.13)]{Dunster:2020:LGE} one can simply equate powers of $\mu$ in the formal expansions
%%%%%%%
\begin{equation} \label{99j}
\sum_{m=1}^{\infty }\frac{\hat{E}_{2m}(z)}{\mu^{2m}}
= -\frac{1}{2}\ln \left\{1+\sum_{m=0}^{\infty}
\frac{\hat{F}_{2m+1}(z)}{\mu^{2m+2}}\right\}.
\end{equation}
%%%%%%

For example, on expanding the RHS for large $\mu$ and equating powers of $\mu^{-4}$ we find that $\hat{E}_{4}(z)=\frac{1}{4}\hat{F}_{1}^{2}(z)-\frac{1}{2}\hat{F}_{3}(z)$. Note that we have taken $C=0$ in the formula referenced, since $\forall m \in \mathbb{N} \cup \{0\}$ $\hat{F}_{2m+1}(z)=o(1)$ as $z \rightarrow \infty$ and so this is also true for $\hat{E}_{2m}(z)$ ($m=1,2,3,\cdots$), in accord with the lower integration limit in (\ref{97}). Incidentally (\ref{99b}) and (\ref{99j}) show that
%%%%%%%
\begin{equation} \label{99k}
\hat{E}_{2m}(0)=0
\quad (m=1,2,3,\cdots).
\end{equation}
%%%%%%

We now apply \cite[Thm. 1.4]{Dunster:2020:LGE} with $\xi = \alpha_{1}$ corresponding to $z=z^{(1)}=0$, and $\xi = \alpha_{2}$ corresponding to $z=z^{(2)}=+\infty$. From our discussion earlier about the branch of $\xi$ we then have $\Re(\alpha_{j})=(-1)^{j}\infty$ as required in the hypothesis of the theorem, and consequently on using (\ref{95}) we obtain LG solutions of (\ref{88}), for arbitrary positive integer $n$, of the form
%%%%
\begin{equation}
\label{106}
w_{n,1}(\mu,z)=\left(\frac{z}{Z}\right)^{1/2}
\exp\left\{\mu\xi+\sum_{s=1}^{n-1}
\frac{\hat{E}_{s}(z)}{\mu^{s}}\right\}
\left\{1+\eta_{n,1}(\mu,z)\right\},
\end{equation}
%%%%%
and
%%%%
\begin{equation}
\label{112}
w_{n,2}(\mu,z)=\left(\frac{z}{Z}\right)^{1/2}
\exp\left\{-\mu \xi+\sum_{s=1}^{n-1}\left(-1\right)^{s}
\frac{\hat{E}_{s}(z)}{\mu^{s}}\right\}
\left\{1+\eta_{n,2}(\mu,z)\right\}.
\end{equation}
%%%%%

Note that with our branch cut $(z/Z)^{1/2} \rightarrow 1$, and hence $\hat{E}_{s}(\infty) \rightarrow 0$ as $z \rightarrow \infty$ in any direction in $\mathbb{C}^{+}$. Here the relative error terms $\eta_{n,j}(\mu,z)$ ($j=1,2$) are $\mathcal{O}(\mu^{-n})$ uniformly for $z \in Z_{j}(\delta)$, with error bounds provided by \cref{thm:MWLG} below. From (\ref{103}) and (\ref{109}) the important property of these functions is that $w_{n,1}(\mu,z) \rightarrow 0$ as $z \rightarrow 0$ ($\Re(\xi) \rightarrow -\infty$) and $w_{n,2}(\mu,z) \rightarrow 0$ as $\Re(z) \rightarrow \infty$ ($\Re(\xi) \rightarrow \infty$).

Matching solutions recessive at $z=0$ implies that $M_{\kappa,\mu}(\mu z)=c_{n,1}(\mu)w_{n,1}(\mu,z)$ for some constant $c_{n,1}(\mu)$. Then using (\ref{104}), (\ref{103}), (\ref{99k}) and (\ref{106}) yields
%%%%
%%%%
\begin{equation}
\label{107}
c_{n,1}(\mu)=
\lim_{z \rightarrow 0}\left\{
\frac{M_{\kappa,\mu}(\mu z)}{w_{n,1}(\mu,z)} \right\}
=\sqrt{2\mu}\left\{2(1-\lambda)\right\}^{\kappa}
\left(\frac{4\mu}{e}\right)^{\mu}
\exp \left\{-\sum_{s=1}^{n-1}
\frac{\hat{E}_{s}(0)}{\mu^{s}}\right\}.
\end{equation}
%%%%%
Note from (\ref{92}) and (\ref{98})
%%%%
\begin{equation}
\label{101}
\hat{E}_{1}(0)=\frac{2-\lambda^{2}}{24\left(1-\lambda^{2}\right)}.
\end{equation}

Similarly, on matching solutions recessive at $z=+\infty$, and using (\ref{110}), (\ref{109}), (\ref{99k}) and (\ref{112}), results in the identification $W_{\kappa,\mu}(\mu z)=c_{2}(\mu)w_{n,2}(\mu,z)$, where
%%%%
\begin{equation}
\label{113}
c_{2}(\mu)=
\lim_{z \rightarrow \infty}\left\{
\frac{W_{\kappa,\mu}(\mu z)}{w_{n,2}(\mu,z)} \right\}
=\frac{1}{(1-\lambda)^{\mu}}
\left(\frac{\mu}{2e}\right)^{\kappa}.
\end{equation}
%%%%%

In summary we have our first main result.
\begin{theorem}
\label{thm:MWLG}
%%%%
As $\mu \rightarrow \infty$, $\lambda=\kappa/\mu \in [0,1-\delta]$, and positive integers $n$ and $r$
%%%%%%%%%%%%%%%
\begin{multline}
\label{117}
M_{\kappa,\mu}(\mu z)=(4\mu)^{\mu}
\left\{2(1-\lambda)\right\}^{\kappa}
\left(\frac{2\mu z}{Z}\right)^{1/2}   \\
\times \exp \left\{\mu (\xi-1)
+\sum_{s=1}^{n-1}\frac{\hat{E}_{s}(z)-\hat{E}_{s}(0)}
{\mu^{s}}\right\}\left\{1+\eta_{n,1}(\mu,z)\right\},
\end{multline}
%%%%%
and
%%%%
\begin{multline}
\label{124}
W_{\kappa,\mu}(\mu z)=\left(\frac{\mu}{2e}\right)^{\kappa}\frac{1}{\left(1-\lambda \right)^{\mu}}\left(\frac{z}{Z}\right)^{1/2}\\
\times \exp \left\{-\mu \xi +\sum_{s=1}^{n-1}\left(-1\right)^{s}\frac{\hat{E}_{s}(z)}{\mu^{s}}\right\}\left\{1+\eta_{n,2}(\mu,z)\right\},
\end{multline}
%%%%%
where for $z \in Z_{1}(\delta)$
%%%%
\begin{multline}
\label{120}
\left| \eta_{n,1}(\mu,z)\right| \leq \left| \exp \left\{\sum_{s=n}^{n+r-1}\frac{\hat{E}_{s}(z)-\hat{E}_{s}(0)}
{\mu^{s}}\right\}-1\right| \\
+\frac{\omega_{n+r,1}(\mu,z)}{\mu^{n+r}}\exp \left\{\frac{\varpi_{n+r,1}(\mu,z)}{\mu}+\sum_{s=n}^{n+r-1}
\frac{\Re\{\hat{E}_{s}(z)\}-\hat{E}_{s}(0)}
{\mu^{s}}+\frac{\omega_{n+r,1}(\mu,z)}{\mu^{n+r}}\right\},
\end{multline}
%%%%%
and for $z \in Z_{2}(\delta)$
%%%%
\begin{multline}
\label{126}
\left| \eta_{n,2}(\mu,z)\right| \leq \left| \exp \left\{\sum_{s=n}^{n+r-1}\left(-1\right)^{s}\frac{\hat{E}_{s}(z)}
{\mu^{s}}\right\}-1\right| \\
+\frac{\omega_{n+r,2}(\mu,z)}{\mu^{n+r}}\exp \left\{\frac{\varpi_{n+r,2}(\mu,z)}{\mu}+\sum_{s=n}^{n+r-1}\left(-1\right)^{s}\frac{\Re\{\hat{E}_{s}(z)\}}
{\mu^{s}}+\frac{\omega_{n+r,2}(\mu,z)}
{\mu^{n+r}}\right\},
\end{multline}
%%%%%
where, for $j=1,2$
%%%%
\begin{equation}
\label{121}
\omega_{n,j}(\mu,z)=\int_{\mathscr{L}_{j}(z)}
\left| \frac{Z\left(t\right)\hat{F}_{n}
\left(t\right)dt}{t}\right| +\sum_{s=1}^{n-1}\frac{1}{2\mu^{s}}\int_{\mathscr{L}_{j}(z)}\left| \frac{Z\left(t\right)\hat{G}_{n,s}\left(t\right)dt}{t}\right|,
\end{equation}
%%%%%
and
%%%%
\begin{equation}
\label{123}
\varpi_{n,j}(\mu,z)=\sum_{s=0}^{n-2}\frac{2}{\mu^{s}}
\int_{\mathscr{L}_{j}(z)}\left|\frac{Z\left(t\right)
\hat{F}_{s+1}\left(t\right)dt}{t}\right|.
\end{equation}
%%%%%
Here $\mathscr{L}_{j}(z)$ is a progressive path linking $z$ to the singularity $z^{(j)}$ ($z^{(1)} =0$, $z^{(2)} =+\infty$), and
%%%%
\begin{equation}
\label{122}
\hat{G}_{n,s}(z)=\sum_{k=s}^{n-1}\hat{F}_{k}(z)
\hat{F}_{s+n-k-1}(z).
\end{equation}
%%%%%
\end{theorem}

Here and elsewhere for each $n$ the value of $r$ can be chosen to sharpen the bounds, at a small price of computing a few extra coefficients $\hat{E}_{s}(z)$. If $r=0$ is chosen then the sums on the RHS of the inequality are understood to be zero.

Both expansions in this theorem hold in domains containing the positive real axis. Recall that $Z_{j}(\delta)$ ($j=1,2$) consists of all points in $\mathbb{C}^{+}$ except those whose distance to the curve $\gamma_{j}$ is less than $\delta$ (see \cref{fig:fig2}). For complex $z$ care must be taken in (\ref{117}) when $z \in S_{1}$: in this case the branch of $\xi$ is not the principal one, but rather the one taken when crossing $\gamma_{2}$ from $S_{-1}$, and as such $\Re(\xi) \geq 0$. Also in this case the branches of $Z$ and the coefficients differ from their principal values. We give the appropriate representation in \cref{thm:MW2} below in terms of the principal values of $\xi$ and the other multi-valued functions. Note that we do not have this situation with (\ref{124}) since there are no branch cuts in $Z_{2}(\delta)$ and all functions take their principal values in this region of asymptotic validity.

An asymptotic expansion for the third fundamental solution $W_{-\kappa,\mu}(\mu ze^{-\pi i})$, which is recessive in $S_{1}$, follows in a similar manner to the previous two expansions. We again have the asymptotic solution given by (\ref{106}), but this time the reference point is taken to be $z=z^{(3)}=\infty e^{\pi i}$. This gives a solution which is recessive at this singularity, and although of a similar form, is independent of $w_{n,2}(\mu,z)$. We call it $w_{n,3}(\mu,z)$, with relative error denoted by $\eta_{n,3}(\mu,z)$.

Now from (\ref{110}) as $z \rightarrow \infty$ in $\mathbb{C}^{+}$
%%%%
\begin{equation}
\label{129a}
W_{-\kappa,\mu}\left(\mu ze^{-\pi i}\right)
=e^{\kappa \pi i}(\mu z)^{-\kappa}e^{\mu z/2}\left\{1+\mathcal{O}\left(z^{-1}\right)\right\},
\end{equation}
%%%%%
and from this and (\ref{109}) we arrive at:
\begin{theorem}
As $\mu \rightarrow \infty$, $\lambda=\kappa/\mu \in [0,1-\delta]$, positive integer $n$ and nonnegative integer $r$
\begin{multline}  \label{129d}
W_{-\kappa,\mu}\left(\mu ze^{-\pi i}\right)
=e^{\kappa\pi i} (1-\lambda)^{\mu}
\left(\frac{2e}{\mu}\right)^{\kappa}
\left(\frac{z}{Z}\right)^{1/2}\\
\times \exp \left\{\mu \xi 
+\sum_{s=1}^{n-1}\frac{\hat{E}_{s}(z)}
{\mu^{s}}\right\}\left\{1+\eta_{n,3}(\mu,z)\right\},
\end{multline}
%%%%%
where for $z \in Z_{3}(\delta)$
%%%%
\begin{multline}
\label{129e}
\left| \eta_{n,3}(\mu,z)\right| \leq \left| \exp \left\{\sum_{s=n}^{n+r-1}\frac{\hat{E}_{s}(z)}
{\mu^{s}}\right\}-1\right| \\
+\frac{\omega_{n+r,3}(\mu,z)}{\mu^{n+r}}\exp \left\{\frac{\varpi_{n+r,3}(\mu,z)}{\mu}
+\sum_{s=n}^{n+r-1}\frac{\Re\{\hat{E}_{s}(z)\}}{\mu^{s}}
+\frac{\omega_{n+r,3}(\mu,z)}{\mu^{n+r}}\right\}.
\end{multline}
%%%%%
Here $\omega_{n+r,3}(\mu,z)$ and $\varpi_{n+r,3}(\mu,z)$ are given by (\ref{121}) and (\ref{123}) with the integration path $\mathscr{L}_{3}(z)$ being a progressive one linking $z$ to $z^{(3)}=\infty e^{\pi i}$.
\end{theorem}
%%%%

As in (\ref{117}) this must be modified with non-principal values of $\xi$ and the coefficients when crossing the cut along $\gamma_{2}$, but this time when $z \in S_{-1}$. We now show how both these expansions can be expressed in terms of principal values in these cases.

Before stating our main result we define sequences of constants $k_{s}$ and $l_{s}$ ($s=0,1,2,\cdots$) that will appear. These are given by $k_{2m}=l_{2m}=0$ ($m=1,2,3,\cdots$), and for $m=0,1,2,\cdots$
%%%%
\begin{equation}
\label{117c}
k_{2m+1}=l_{2m+1}-\hat{E}_{2m+1}(0),
\end{equation}
%%%%%
in which
%%%%
\begin{equation}
\label{129bb}
l_{2m+1}=\lim_{z \rightarrow -\infty}\hat{E}_{2m+1}(z)
\quad   (m=0,1,2,\cdots),
\end{equation}
%%%%%
where in this limit we cross the cut $\gamma_{2}$, and so must use $Z \sim |z|$ instead of $Z \sim z$. For example from (\ref{98}) we have
%%%%
\begin{equation}
\label{129cc}
l_{1}=\lim_{z \rightarrow -\infty}
\frac{\lambda z^{3}}{24\left(1-\lambda^{2}\right)|z|^{3}}
-\frac{\lambda}{24\left(1-\lambda^{2}\right)}
=-\frac{\lambda}{12\left(1-\lambda^{2}\right)},
\end{equation}
%%%%%
and hence from (\ref{101}) and (\ref{117c})
%%%%
\begin{equation}
\label{117d}
k_{1}=
-\frac{\lambda}{12\left(1-\lambda^{2}\right)}
-\hat{E}_{1}(0)=\frac{2+\lambda}{24(1+\lambda)}.
\end{equation}

Generally from (\ref{99f}), (\ref{99g}) and (\ref{99hh}) one can show that
%%%
\begin{equation}
\label{129dd}
l_{2m+1}=-\frac{1}{2^{6m+5}\left(1-\lambda^{2}\right)^{3m+2}}
\int_{-1}^{1} \left(1-\beta^{2}\right)^{3m+1}
p_{4m+3}^{(e)}(\tau) d\beta,
\end{equation}
%%%%%
for $m=0,1,2,\cdots$, where $\tau=\tau(\beta)$ is given by (\ref{99ccc}). Note that the integrand is a polynomial in $\beta$.

%%%%

\begin{theorem}
\label{thm:MW2}
As $\mu \rightarrow \infty$, $\lambda=\kappa/\mu \in [0,1-\delta]$, positive integer $n$ and nonnegative integer $r$, $\xi$ and the coefficients taking their principal values, we have for $z \in Z_{1}(\delta) \cap S_{1}$
%%%%%%%%%%%%%%
\begin{multline}
\label{117b}
M_{\kappa,\mu}(\mu z)=
\frac{ie^{(\mu-\kappa)\pi i}}
{\left(1-\lambda^2\right)^{\mu}
\left\{2(1+\lambda)\right\}^{\kappa}}
\left(\frac{4\mu}{e}\right)^{\mu}
\left(\frac{2\mu z}{Z}\right)^{1/2}    \\
\times \exp \left\{-\mu \xi
+\sum_{s=1}^{n-1}(-1)^{s}\frac{\hat{E}_{s}(z) 
-k_{s}}{\mu^{s}}\right\}
\left\{1+\eta_{n,4}(\mu,z)\right\},
\end{multline}
%%%%%%%%%%%%%%%%%%
and for $z \in Z_{3}(\delta) \cap S_{-1}$
%%%%
\begin{multline}  \label{129}
W_{-\kappa,\mu}\left(\mu ze^{-\pi i}\right)
=-\frac{ie^{\mu \pi i}}{\left(1+\lambda \right)^{\mu}}
\left\{\frac{e}{2\mu\left(1-\lambda^{2}\right)}\right\}^{\kappa}
\left(\frac{z}{Z}\right)^{1/2}\\
\times \exp \left\{-\mu \xi 
+\sum_{s=1}^{n-1}(-1)^{s}\frac{\hat{E}_{s}(z)-l_{s}}
{\mu^{s}}\right\}\left\{1+\eta_{n,5}(\mu,z)\right\},
\end{multline}
%%%%%%%%%%
with error bounds
%%%%
\begin{multline}
\label{117bb}
\left| \eta_{n,4}(\mu,z)\right| \leq \left| \exp \left\{\sum_{s=n}^{n+r-1}\left(-1\right)^{s}
\frac{\hat{E}_{s}(z)-k_{s}}
{\mu^{s}}\right\}-1\right| \\
+\frac{\omega_{n+r,1}(\mu,z)}{\mu^{n+r}}\exp \left\{\frac{\varpi_{n+r,1}(\mu,z)}{\mu}
+\sum_{s=n}^{n+r-1}\left(-1\right)^{s}
\frac{\Re\{\hat{E}_{s}(z)\}-k_{s}}
{\mu^{s}}+\frac{\omega_{n+r,1}(\mu,z)}
{\mu^{n+r}}\right\},
\end{multline}
%%%%%
and
%%%%
\begin{multline}
\label{131}
\left| \eta_{n,5}(\mu,z)\right| \leq \left| \exp \left\{\sum_{s=n}^{n+r-1}(-1)^{s}\frac{\hat{E}_{s}(z)-l_{s}}
{\mu^{s}}\right\}-1\right| \\
+\frac{\omega_{n+r,3}(\mu,z)}{\mu^{n+r}}\exp \left\{\frac{\varpi_{n+r,3}(\mu,z)}{\mu}
+\sum_{s=n}^{n+r-1}(-1)^{s}\frac{\Re\{\hat{E}_{s}(z)-l_{s}\}}{\mu^{s}}
+\frac{\omega_{n+r,3}(\mu,z)}{\mu^{n+r}}\right\},
\end{multline}
%%%%%
where for $j=1,3$ $\omega_{n+r,j}(\mu,z)$ and $\varpi_{n+r,j}(\mu,z)$ are given by (\ref{121}) and (\ref{123}) with progressive integration paths $\mathscr{L}_{j}(z)$ linking $z$ to $z^{(j)}$.
\end{theorem}

\begin{proof}
Consider first (\ref{117}), and assume that $z \in S_{1}$ crossing $\gamma_{2}$ from $S_{-1}$. We must replace $Z$ by $-Z$ in all terms. Thus from (\ref{91}) on this sheet $\xi$ must be replaced by
%%%%
\begin{equation}
\label{117a}
-\xi-(1+\lambda)\ln\left(1-\lambda^2\right)
-2\lambda\ln(2)+(1-\lambda)\pi i,
\end{equation}
%%%%%
where here $\xi$ takes its principal value. All the even coefficients $\hat{E}_{2m}(z)$ ($m = 1,2,3, \allowbreak \cdots$) are single-valued (meromorphic at $z=z^{+}(\lambda)$) and so are unchanged. However, the odd coefficients are of opposite sign and differ by a constant. Specifically, from (\ref{97}) and (\ref{99b}) one can show that on this sheet $\hat{E}_{2m+1}(z)$ ($m=0,1,2,\cdots$) must be replaced by
%%%%
\begin{equation}
\label{117aa}
-\hat{E}_{2m+1}(z)+\lim_{z \rightarrow -\infty}
\hat{E}_{2m+1}(z),
\end{equation}
%%%%%
where, as in (\ref{129bb}), we must use $Z \sim |z|$ rather than $Z \sim z$ in this limit. The expansion (\ref{117b}) then follows from (\ref{117}), (\ref{120}) (with $\eta_{n,1}(\mu,z)$ labeled $\eta_{n,4}(\mu,z)$), (\ref{117c}), (\ref{129bb}), (\ref{117a}) and (\ref{117aa}). The expansion (\ref{129}) is proved similarly.
\end{proof}

Next from the connection formula \cite[Eq. 13.14.32]{NIST:DLMF}
%%%%
\begin{equation}
\label{55}
M_{\kappa,\mu}(\mu z)
=\frac{ie^{\left(\mu -\kappa \right)\pi i}
\Gamma \left(2\mu +1\right)}
{\Gamma \left(\mu +\kappa +\frac{1}{2}\right)}
W_{\kappa,\mu}(\mu z)
+\frac{e^{-\kappa \pi i}\Gamma \left(2\mu +1\right)}
{\Gamma \left(\mu -\kappa +\frac{1}{2}\right)}
W_{-\kappa,\mu}\left(\mu ze^{-\pi i}\right),
\end{equation}
%%%%%
one can extend the results of \cref{thm:MWLG,thm:MW2} to all values of $z$ except for points within a distance $\mathcal{O}(\delta)$ of the turning point $z^{+}(\lambda)$. For example plugging (\ref{117}) and (\ref{124}) into (\ref{55}) yields an expansion for $W_{-\kappa,\mu}(\mu ze^{-\pi i})$ which is valid for $z \in Z_{1}(\delta) \cap Z_{2}(\delta)$. Thus this compound expansion, along with (\ref{129d}), covers all $z \in \{Z_{1}(\delta) \cap Z_{2}(\delta)\} \cup Z_{3}(\delta)$, which is equivalent to all $z \in \mathbb{C}^{+}$ except for an $\mathcal{O}(\delta)$ neighbourhood of the turning point.

This connection formula also provides a method for computing the nonzero constants $k_{2m+1}$. To do so we plug (\ref{124}) and (\ref{117b}) into (\ref{55}), and let $\Re(z) \rightarrow -\infty$. In this limit $W_{-\kappa,\mu}(\mu ze^{-\pi i})$ vanishes, with the other two being exponentially large. On recalling (\ref{97a}) and equating these dominant terms yields the formal expansion
%%%%
\begin{equation}
\label{117e}
\left(\frac{e}{4\mu}\right)^{\mu}
\left(\frac{\mu}{e}\right)^{\lambda\mu}
\frac{(1+\lambda)^{\mu +\lambda\mu}\Gamma \left(2\mu +1\right)}
{\sqrt{2\mu}\,\Gamma \left(\mu +\lambda\mu +\frac{1}{2}\right)}
\sim
\exp \left\{\sum_{m=0}^{\infty}
\frac{k_{2m+1}}{\mu^{2m+1}}\right\}
\quad (\mu \rightarrow \infty).
\end{equation}
%%%%%%%
As in (\ref{99j}) we take logarithms of both sides, expand the LHS in inverse powers of large $\mu$, and then equate powers to find these coefficients. For the gamma function expansions Stirling's formula \cite[Eq. 5.11.1]{NIST:DLMF} must be used.

\subsection{Numerical examples}

\begin{figure}[hthp]
  \centering
  \includegraphics[trim={0 100 0 100}, width=0.7\textwidth,keepaspectratio]{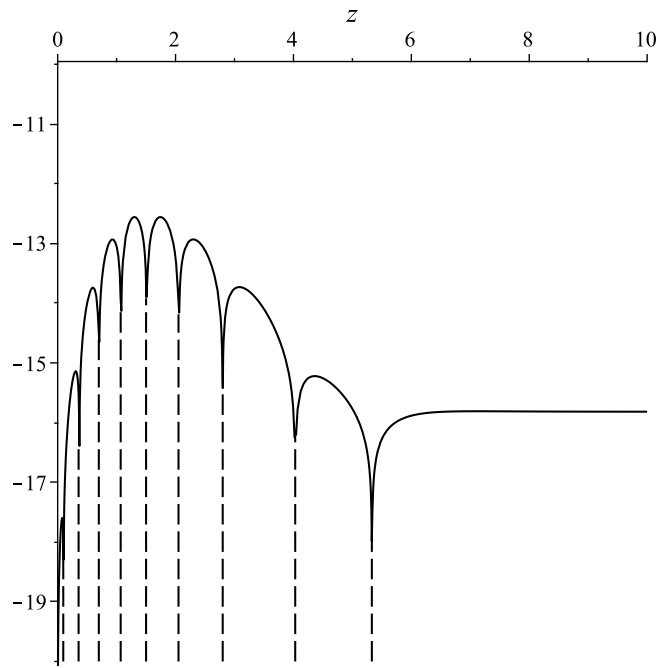}
  \caption{Graph of $\log_{10}|\eta_{11,1}(\mu,z)|$ for $\mu=20$ and $\kappa=4.5$}
  \label{fig:fig3}
\end{figure}

\begin{figure}[hthp]
  \centering
  \includegraphics[trim={0 100 0 100}, width=0.7\textwidth,keepaspectratio]{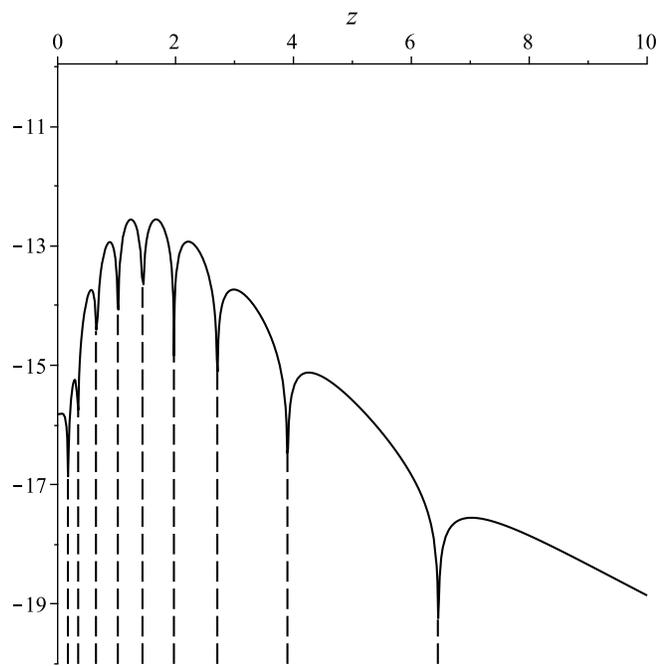}
  \caption{Graph of $\log_{10}|\eta_{11,2}(\mu,z)|$ for $\mu=20$ and $\kappa=4.5$}
  \label{fig:fig4}
\end{figure}

In \cref{fig:fig3} ($j=1$) and \cref{fig:fig4} ($j=2$) graphs are given for the exact values of $\log_{10}|\eta_{n,j}(\mu,z)|$ appearing in \cref{thm:MWLG} for $n=11$, $\mu=20$, $\kappa=4.5$ and $0<z\leq 10$. The exact values of the Whittaker functions were computed in Maple\footnote{Maple 2020. Maplesoft, a division of Waterloo Maple Inc., Waterloo, Ontario} with Digits set at 30. The values of $M_{\kappa,\mu}(\mu z)$ were checked with quadrature using the following integral representation \cite[Eq. 13.16.1]{NIST:DLMF}
\begin{multline}
\label{142}
M_{\kappa,\mu}(\mu z)
=\frac{\Gamma\left(1+2\mu\right)
(\mu z)^{\mu+\frac{1}{2}}}{2^{2\mu}
\Gamma\left(
\mu-\kappa+\frac{1}{2}\right)
\Gamma\left(\mu+\kappa+\frac{1}{2}\right)}  \\  \times
\int_{-1}^{1}e^{\mu zt/2}(1+t)^{\mu-\frac{1}{2}
-\kappa}(1-t)^{\mu+\kappa-\frac{1}{2}}dt
\quad (\Re(\mu \pm \kappa)>-\tfrac{1}{2}).
\end{multline}
%%%%%
Likewise, since $N:=\mu +\kappa -\frac{1}{2}=24$ is an integer for our choices of $\mu$ and $\kappa$ we were able to check the exact value of $W_{\kappa,\mu}(\mu z)$ using the finite sum \cite[Eq. 13.14.9]{NIST:DLMF}
%%%%
\begin{multline}
\label{141}
W_{\kappa,\mu}(\mu z)=\frac{\Gamma \left(\mu +\kappa +\frac{1}{2}\right)
(\mu z)^{\kappa}e^{-\mu z/2}}{\Gamma \left(\mu -\kappa +\frac{1}{2}\right)}
\\ \times
\sum_{s=0}^{N}\frac{\Gamma \left(\mu -\kappa +s+\frac{1}{2}\right)}{\Gamma \left(\mu +\kappa -s+\frac{1}{2}\right)s!(\mu z)^{s}}
\quad (z \neq 0).
\end{multline}
%%%%%

From these graphs we notice the uniformity of the expansions throughout the interval. They also indicate as expected that $\eta_{n,1}(\mu,z) \rightarrow 0$ as $z \rightarrow 0$ and $\eta_{n,2}(\mu,z) \rightarrow 0$ as $z \rightarrow \infty$. The peak (least accuracy) in both graphs occurs between $z=1$ and $z=2$, and this is explained by the presence of the turning points which lie in the first and fourth quadrants on the circle $|z|=2$. The dashed lines are vertical asymptotes, where the relative error is zero and hence the logarithm is $-\infty$.

\section{Turning point expansions} 
\label{sec3}
In this section we construct asymptotic expansions that are valid at the turning point, thus in conjunction with the previous LG expansions cover the entire upper half plane $\mathbb{C}^{+}$. In fact the expansions here, which involve Airy functions, are themselves uniformly valid for all $z \in \mathbb{C}^{+}$. However away from the turning point, in particular on the positive real axis, the simpler LG expansions are usually preferable.

We begin by defining a new LG variable which we label $\xi^{+}$, along with the turning point variable $\zeta$ by (c.f. (\ref{91}))
\begin{multline}  \label{62}
\xi^{+}=\frac{2}{3}\zeta^{3/2}
=\int_{z^{+}(\lambda)}^{z} f^{1/2}(t)dt 
=\tfrac{1}{2}Z
-\lambda \ln \left\{\frac{1}{2}(Z+z-2\lambda)\right\} \\
-\ln \left(\frac{Z-\lambda z+2}{z}\right)
+\frac{1}{2}(1+\lambda)\ln \left(1-\lambda^2\right)
-\frac{1}{2}(1-\lambda)\pi i.
\end{multline}
Both of these appear in our Airy functions expansions, and of course $\xi^{+}$ only differs from $\xi$ by an additive constant. The turning point at $z=z^{+}(\lambda)$ is mapped to $\zeta=\xi=0$. Furthermore $\zeta $ is an analytic function of $z$ in $\mathbb{C}^{+}$, and also in part of the lower half plane (details of which do not concern us).

The branches for $\zeta$ are taken such that as $z \rightarrow z^{+}(\lambda)$
%%%%
\begin{equation}
\label{63}
\zeta = 2^{-2/3} \left(1-\lambda^2\right)^{1/6}
e^{i \theta^{+}(\lambda)}
\left(z-z^{+}(\lambda)\right)
+\mathcal{O}\left\{\left(z-z^{+}(\lambda)\right)^{2}\right\},
\end{equation}
%%%%%
where $(1-\lambda^2)^{1/6}>0$,
%%%%
\begin{equation}
\label{63a}
\theta^{+}(\lambda) =
\tfrac{1}{3}\arcsin\left(2\lambda^{2}-1\right)
\in \left[-\tfrac{1}{6}\pi,\tfrac{1}{6}\pi\right)
\quad (0\leq \lambda <1),
\end{equation}
%%%%%
and by continuity elsewhere in $\mathbb{C}^{+}$. The branch for $\xi^{+}$ is the same as $\xi$, namely that it is an analytic function of $z$ in $\mathbb{C}^{+} \setminus \gamma_{2}$, with $\Re(\xi^{+}) \rightarrow \infty$ as $\Re(z) \rightarrow \infty$.

From (\ref{62}) we have
%%%%
\begin{multline}
\label{65}
\tfrac{2}{3}\zeta^{3/2} =
\ln \left(\tfrac{1}{4}z\right)+1-\lambda \ln(1-\lambda)
+\tfrac{1}{2}(1+\lambda)\ln \left(1-\lambda^2\right) \\
-\tfrac{1}{2}(1-\lambda)\pi i + \mathcal{O}(z)
\quad (z \rightarrow 0),
\end{multline}
%%%%%
and
%%%%
\begin{multline}
\label{64}
\tfrac{2}{3}\zeta^{3/2} =\tfrac{1}{2}z
-\lambda\ln(z)-\lambda -\ln(1-\lambda)
+\tfrac{1}{2}(1+\lambda)\ln \left(1-\lambda^2\right)  \\
-\tfrac{1}{2}(1-\lambda)\pi i+\mathcal{O}\left(z^{-1}\right)
\quad (z \rightarrow \infty).
\end{multline}
%%%%%
Thus from (\ref{65}) and (\ref{64}) we observe that $\zeta \rightarrow \infty$ in the sectors $$\mathbf{S}_{j}:=\{\zeta: |\arg(\zeta e^{-2\pi i j/3})|\leq \pi /3\}$$ as $z \rightarrow z^{(1)}$ ($j=-1$), $z \rightarrow z^{(2)}$ ($j=0$), and $z \rightarrow z^{(3)}$ ($j=1$). Also note that the curve $\gamma_{1}$ corresponds to the line $\arg(\zeta)=0$, and the curves $\gamma_{j}$ ($j=2,3$) correspond to part of the lines $\arg(\zeta)=\pm 2 \pi /3$. Since the $z-\zeta$ mapping is conformal at the turning point it follows that these three level curves meet at an angle $2\pi/3$ with one another at this point of intersection.

In our turning point expansions we find from \cite[Thm. 2.1]{Dunster:2017:COA} that the odd coefficients need to be such that $(z-z^{+}(\lambda))^{1/2}E_{2m+1}(z)$ ($m=0,1,2,\cdots$) are meromorphic at $z=z^{+}(\lambda)$. The odd coefficients of \cref{sec2} do not meet this criterion, but this is achieved by modifying (\ref{99f}) to the slightly different coefficients defined by
%%%
\begin{equation}
\label{99i}
E_{2m+1}^{+}(z)
=\frac{1}{8\left(1-\lambda^{2}\right)}
\left\{\int_{0}^{\beta} \frac{p_{4m+3}^{(e)}(\tau(b))}
{\tilde{Z}^{3s-1}(b)}db
+\int_{1}^{\beta} \frac{p_{4m+3}^{(o)}(\tau(b))}
{\tilde{Z}^{3s-1}(b)}db\right\},
\end{equation}
%%%%%
where $\tau(b)$ is given by (\ref{99ccc}) with $\beta=b$. From (\ref{99g}) the first integral results in an odd polynomial in $\beta$, which from (\ref{92}), (\ref{99c}) and (\ref{99cc}) is of the desired form. 

To show that the same is true of the second integral, we make the $u$ substitution as discussed after (\ref{99b}), and this leads to a sum of integrals of the form
%%%%%%%
\begin{equation} 
\label{99ii}
\mathrm{constant}\times\int_{0}^{1-\beta^2}(1-u)^{s}
u^{3m-s-\frac{1}{2}}du
\quad (s=0,1,2,\cdots 2m+1).
\end{equation}
%%%%%%
Now from (\ref{92}), (\ref{99c}) and (\ref{99cc})
%%%%%%%
\begin{equation} 
\label{99jj}
\left(1-\beta^{2}\right)^{\frac{1}{2}}
=\frac{2\left(1-\lambda^{2}\right)^{\frac{1}{2}}}{Z}.
\end{equation}
%%%%%%
Hence on expanding the integrated terms of (\ref{99ii}), and referring to (\ref{99c}) and (\ref{99cc}) again, results in sums of terms of the form
%%%%%%%
\begin{equation} 
\label{99kk}
\mathrm{constant}\times\beta^{2r}
\left(1-\beta^{2}\right)^{\frac{1}{2}}
=\mathrm{constant}\times
\frac{(z-2\lambda)^{2r}}{Z^{2r+1}}
\quad (r\in \mathbb{N}),
\end{equation}
%%%%%%
which is also of the desired form. 

Of course $\hat{E}_{2m+1}^{+}(z)$ and $\hat{E}_{2m+1}(z)$ only differ by an additive constant, namely $\hat{E}_{2m+1}^{+}(\infty)$ (see (\ref{97a})). For example, from (\ref{98}) we see that
%%%%%%%
\begin{equation} \label{99ll}
\hat{E}_{1}^{+}(z)=
\hat{E}_{1}(z)+\frac{\lambda}{24\left(1-\lambda^{2}\right)}
=\frac{\lambda z^{3}+6\left(1-2\lambda^{2}\right)
z^{2}+12\lambda^{3}z+8\lambda^{2}-16}
{24\left(1-\lambda^{2}\right)Z^{3}}.
\end{equation}
%%%%%%

We remark, unlike their odd counterparts, the even coefficients in our turning point expansions do not need modifying, and so
%%%%%
\begin{equation} 
\label{99m}
\hat{E}_{2m}^{+}(z)=\hat{E}_{2m}(z)
\quad (m=1,2,3,\cdots),
\end{equation}
%%%%%%
and hence from (\ref{97a}) $\hat{E}_{2m}^{+}(\infty)=0$. The odd coefficients do not in general vanish at infinity, however the branch cut $\gamma_{2}$ is the same as for the coefficients in the previous section, and hence $\lim_{z \rightarrow \infty}\hat{E}_{2m+1}^{+}(z)=\hat{E}_{2m+1}^{+}(\infty)$ is independent of direction in the upper half plane of this limit.

We next define coefficient functions from \cite[Eqs. (1.16) - (1.18)]{Dunster:2020:SEB}. These involve two sequences of numbers given by $\left\{a_{s}\right\}_{s=1}^{\infty}$ and $\left\{\tilde{a}_{s}\right\}_{s=1}^{\infty}$, where $a_{1}=a_{2}=\frac{5}{72}$, $\tilde{a}_{1}=\tilde{a}_{2}=-\frac{7}{72}$, and subsequent terms $a_{s}$ and $\tilde{{a}}_{s}$ ($s=3,4,5,\cdots $) satisfying the same recursion formula, namely
\begin{equation}  
\label{73}
b_{s+1}=\frac{1}{2}\left(s+1\right) b_{s}+\frac{1}{2}
\sum\limits_{j=1}^{s-1}{b_{j}b_{s-j}} 
\quad (s=2,3,4,\cdots).
\end{equation}
%%%%%%%%%%%%%%%%%%%%
Then for $s=1,2,3,\cdots$ define
\begin{equation}  
\label{74}
\mathcal{E}_{s}(z) =\hat{E}_{s}^{+}(z) +
(-1)^{s}a_{s}s^{-1}(\xi^{+})^{-s},
\end{equation}
and
\begin{equation}  
\label{75}
\tilde{\mathcal{E}}_{s}(z) =\hat{E}_{s}^{+}(z)
+(-1)^{s}\tilde{a}_{s}s^{-1}(\xi^{+})^{-s}.
\end{equation}

From \cite[Thm. 3.4]{Dunster:2020:SEB} three asymptotic solutions of (\ref{88}) are given by $w_{m,j}(\mu,z)$, where
\begin{equation}  
\label{79}
w_{m,j}(\mu,z)=\mathrm{Ai}_{j} 
\left(\mu^{2/3}\zeta \right)\mathcal{A}_{2m+2}(\mu,z)
+\mathrm{Ai}'_{j} \left(\mu^{2/3}\zeta \right)
\mathcal{B}_{2m+2}(\mu,z)
\quad (j=0,\pm 1),
\end{equation}
in which the Airy functions of complex argument are defined by
\begin{equation}  
\label{79b}
\mathrm{Ai}_{j}(\zeta)=\mathrm{Ai}(\zeta e^{-2\pi i j/3})
\quad (j=0,\pm 1).
\end{equation}
These are recessive in the sectors $\mathbf{S}_{j}$; see \cite[Chap. 11, Sect. 8.1]{Olver:1997:ASF}.

The coefficient functions $\mathcal{A}_{2m+2}(\mu,z)$ and $\mathcal{B}_{2m+2}(\mu,z)$ are slowly varying in $\mu$ and are analytic in $z \in \mathbb{C}^{+}$, and as $\mu \rightarrow \infty$ they possess the asymptotic expansions
\begin{equation}  
\label{76}
\mathcal{A}_{2m+2}(\mu,z) = \phi(z)\exp \left\{ \sum\limits_{s=1}^{m}\frac{
\tilde{\mathcal{E}}_{2s}(z)}{\mu^{2s}}\right\} \cosh \left\{ \sum\limits_{s=0}^{m}\frac{\tilde{\mathcal{E}}_{2s+1}(z)}{\mu^{2s+1}}\right\}\left\{1+\tilde{\eta}_{2m+2}(\mu,z)\right\},
\end{equation}
and
\begin{equation}  
\label{77}
\mathcal{B}_{2m+2}(\mu,z) = \frac{\phi(z)}{\mu^{1/3}\zeta^{1/2}}
\exp \left\{ \sum\limits_{s=1}^{m}\frac{
\mathcal{E}_{2s}(z)}{\mu^{2s}}\right\} \sinh \left\{ \sum\limits_{s=0}^{m}\frac{\mathcal{E}_{2s+1}(z)}
{\mu^{2s+1}}\right\}\left\{1+\eta_{2m+2}(\mu,z)\right\},
\end{equation}
where
\begin{equation}  
\label{78}
\phi(z)=z^{1/2}\left(\frac{\zeta}{z^{2}-4\lambda z+4}\right)^{1/4}
=\left(\frac{z}{Z}\right)^{1/2}\zeta^{1/4}.
\end{equation}
Note that $\phi(z)$ has a removable singularity at the turning point $z=z^{+}(\lambda)$ since $\zeta$, as defined by (\ref{62}), has a simple zero at this point.

The relative errors $\tilde{\eta}_{2m+2}(\mu,z)$ and $\eta_{2m+2}(\mu,z)$ are $\mathcal{O}(\mu^{-2m-2})$ as $\mu \rightarrow \infty$, $\lambda \in [0,1-\delta]$, uniformly for $z \in \mathbb{C}^{+}$. Bounds for them are provided in \cite[Thms. 3.4 and 4.2]{Dunster:2020:SEB}. For $j=0,\pm 1$ the solution $w_{m,j}(\mu,z)$ is characterised as being recessive in $S_{j}$, just like the three Whittaker functions we are approximating.

As shown in \cite{Dunster:2017:COA} in an application to Bessel functions, expansions of the form (\ref{76}) and (\ref{77}) are highly accurate and can be computed without difficulty. Note that $\mathcal{E}_{2s}(z)$ and $\tilde{\mathcal{E}}_{2s}(z)$ are not analytic at  $z=z^{+}(\lambda)$, thus the expansions cannot be used directly at or near this point. Since $\mathcal{A}(\mu,z)$ and $\mathcal{B}(\mu,z)$ are analytic at the turning point there are two options to compute them when $z$ is close to this point. Firstly, if many terms are required for high accuracy, Cauchy's integral formula can be used, as described in \cite[Eq. (2.31)]{Dunster:2017:COA} and \cite[Thm. 4.2]{Dunster:2020:SEB}. However, if a few terms are required then (\ref{76}) and (\ref{77}) can be expanded in a standard asymptotic series involving inverse powers of $\mu^{2}$. Each coefficient in these has a removable singularity at $z=z^{+}(\lambda)$ and can be computed via a Taylor series.

We now state our main result of this section.
\begin{theorem}
\label{thm:MWAiry}
As $\mu \rightarrow \infty$, $\lambda=\kappa/\mu \in [0,1-\delta]$, positive integer $m$, and $z \in \mathbb{C}^{+}$
%%%%
\begin{multline}
\label{54}
M_{\kappa,\mu}(\mu z)
=\frac{2\sqrt{\pi}(1 -\lambda)^{\mu}
\Gamma\left(2\mu+1\right)}
{\mu^{\kappa-\frac{1}{6}}\left( 1-\lambda^{2} \right) ^{(\mu+\kappa)/2}
\Gamma\left(\mu-\kappa+\tfrac{1}{2}\right)}
\\
\times \exp\left\{\frac{1}{2}
\left(\mu-\kappa+\frac{1}{3}\right) \pi i 
+\kappa -\sum\limits_{s=0}^{m}\frac{\hat{E}_{2s+1}^{+}(\infty)}
{\mu^{2s+1}}  \right\}
\left\{1+\delta_{2m+2}^{+}(\mu)\right\}^{-1} w_{m,-1}(\mu,z),
\end{multline}
%%%%%
%%%%
\begin{multline}
\label{50}
W_{\kappa,\mu}(\mu z)=2\sqrt {\pi}\mu^{\kappa+\frac{1}{6}} 
\left( 1-\lambda^{2} \right) ^{\kappa/2}
\left( \frac{1+\lambda}{1 -\lambda} \right) ^{\mu/2}
\\
\times \exp\left\{\frac{1}{2}\left(\kappa -\mu \right) \pi i 
-\kappa +\sum\limits_{s=0}^{m}\frac{\hat{E}_{2s+1}^{+}(\infty)}
{\mu^{2s+1}}  \right\}
\left\{1+\delta_{2m+2}^{-}(\mu)\right\}^{-1} w_{m,0}(\mu,z),
\end{multline}
%%%%%
and
%%%%
\begin{multline}
\label{52}
W_{-\kappa,\mu}\left(\mu ze^{-\pi i}\right)
=\frac{2\sqrt {\pi} (1 -\lambda)^{\mu}}
{\mu^{\kappa-\frac{1}{6}}\left( 1-\lambda^{2} \right) ^{(\mu+\kappa)/2}}
\\
\times \exp\left\{\frac{1}{2}
\left(\mu+\kappa-\frac{1}{3}\right) \pi i 
+\kappa -\sum\limits_{s=0}^{m}\frac{\hat{E}_{2s+1}^{+}(\infty)}
{\mu^{2s+1}}  \right\}
\left\{1+\delta_{2m+2}^{+}(\mu)\right\}^{-1} w_{m,1}(\mu,z),
\end{multline}
%%%%%
where
\begin{multline}  
\label{51}
|\delta_{2m+2}^{\pm}(\mu)| \leq
\max\left\{|\tilde{\eta}_{2m+2}(\mu,\infty)|,
|\eta_{2m+2}(\mu,\infty)|\right\}  \\  \times
\left[1+\exp\left\{\pm 2\sum\limits_{s=0}^{m}
\frac{\hat{E}_{2s+1}^{+}(\infty)}
{\mu^{2s+1}}\right\}\right].
\end{multline}
Here $\mu$ is assumed to be sufficiently large so that $|\delta_{2m+2}^{\pm}(\mu)|$ are less than $1$.
\end{theorem}

\begin{proof}
The three identifications come from matching recessive solutions at the singularities $z=z^{(j)}$. All proportionality constants are found by comparing both sides of the equations as $\Re(z) \rightarrow \infty$. Therefore, by matching solutions that are recessive at $z=z^{(1)}=0$ we have $M_{\kappa,\mu}(\mu z)=c_{m,-1}(\mu)w_{m,-1}(\mu,z)$, and the constant can be determined via
\begin{equation} 
\label{52a}
c_{m,-1}(\mu)=\lim_{\Re(z) \rightarrow \infty}
\frac{M_{\kappa,\mu}(\mu z)}{w_{m,-1}(\mu,z)}.
\end{equation}
%%%%%%%
Now using (\ref{79}) - (\ref{78}) and \cite[Eq. 13.19.2]{NIST:DLMF}
%%%%%%
\begin{equation}  
\label{53}
M_{\kappa,\mu}(z) \sim
\frac{\Gamma\left(2\mu+1\right)z^{-\kappa}
e^{z/2}}{\Gamma\left(\mu-\kappa+\tfrac{1}{2}\right)}
\quad  (\Re(z) \rightarrow \infty),
\end{equation}
%%%%%%
along with
\begin{equation}  
\label{53a}
\mathrm{Ai}_{-1}\left( \mu^{2/3}\zeta \right) \sim
\frac{\exp\left(\mu\xi^{+}-\frac{1}{6}\pi i\right)}
{2\pi ^{1/2}\mu^{1/6}\zeta^{1/4}}
\quad  (\Re(\zeta) \rightarrow \infty),
\end{equation}
which comes from (\ref{79b}) and \cite[Eq. 9.7.5]{NIST:DLMF}, and recalling that $\hat{E}_{2s}^{+}(\infty)=0$, we find $c_{m,-1}$ and arrive at the stated result. The identifications (\ref{50}) and (\ref{52}) follow similarly by matching solutions recessive at $z^{(2)}$ and $z=z^{(3)}$, and again finding the proportionality constants by comparing both sides as $\Re(z) \rightarrow \infty$.
\end{proof}

Next using (\ref{55}), \cref{thm:MWAiry} and \cite[Eq. (22)]{Dunster:2020:ASI}
%%%%%
\begin{equation}
\label{56}
w_{m,0}(u,z)+e^{-2\pi i/3}w_{m,1}(u,z)
+e^{2\pi i/3}w_{m,-1}(u,z)=0,
\end{equation}
%%%%%
yields the formal expansion as $\mu \rightarrow \infty$
\begin{equation}
\label{57}
\left(\frac{1-\lambda}{1+\lambda}\right)^{\mu}
\left\{\frac{e^2}{\mu^2\left( 1-\lambda^{2} \right)}
\right\}^{\lambda \mu}
\frac{\Gamma \left(\mu +\lambda \mu +\frac{1}{2}\right)}
{\Gamma \left(\mu -\lambda \mu +\frac{1}{2}\right)}
\sim
\exp\left\{2\sum\limits_{s=0}^{\infty}
\frac{\hat{E}_{2s+1}^{+}(\infty)}
{\mu^{2s+1}}\right\}.
\end{equation}
%%%%%
On taking the logarithm of both sides, then expanding the LHS as an asymptotic series in inverse powers of $\mu$ (see \cite[Eq. 5.11.1]{NIST:DLMF}), and finally matching like powers of this parameter, yields an explicit expression for each constant $\hat{E}_{2s+1}^{+}(\infty)$.

Finally, if one is willing to forgo error bounds, somewhat more compact asymptotic expansions come from replacing the constants $\hat{E}_{2s+1}^{+}(\infty)$ in \cref{thm:MWAiry} by utilising (\ref{57}). The result reads as follows.

\begin{theorem}
As $\mu \rightarrow \infty$ with $\lambda=\kappa/\mu \in [0,1-\delta]$
%%%%
\begin{equation}
\label{59}
M_{\kappa,\mu}(\mu z)  \sim
\frac{2\sqrt {\pi} \exp\left\{\frac{1}{2}
\left(\mu-\kappa+\frac{1}{3}\right) \pi i\right\}\mu^{1/6}
\Gamma(2\mu+1) w_{-1}(\mu, z)}
{\left\{\Gamma \left(\mu +\kappa +\tfrac{1}{2}\right)
\Gamma \left(\mu -\kappa +\tfrac{1}{2}\right)\right\}^{1/2}},
\end{equation}
%%%%
\begin{equation}
\label{58}
W_{\kappa,\mu}(\mu z) \sim
2\sqrt {\pi} 
\exp\left\{\frac{1}{2}\left(\kappa-\mu\right) \pi i\right\} 
\mu^{1/6} \left\{\frac{\Gamma \left(\mu +\kappa +\frac{1}{2}\right)}
{\Gamma \left(\mu -\kappa +\frac{1}{2}\right)}\right\}^{1/2}
w_{0}(\mu,z),
\end{equation}
%%%%%
%%%%%
and
%%%%
\begin{equation}
\label{59aa}
W_{-\kappa,\mu}\left(\mu ze^{-\pi i}\right) \sim
2\sqrt {\pi} 
\exp\left\{\frac{1}{2}
\left(\mu+\kappa-\frac{1}{3}\right) \pi i\right\}
\mu^{1/6} 
\left\{\frac{\Gamma \left(\mu -\kappa +\frac{1}{2}\right)}
{\Gamma \left(\mu +\kappa +\frac{1}{2}\right)}\right\}^{1/2}w_{1}(\mu,z),
\end{equation}
%%%%%
uniformly for $z \in \mathbb{C}^{+}$, where $w_{j}(\mu,z)$ ($j=0,\pm 1$) are given by (\ref{79}) - (\ref{78}) with $\eta_{2m+2}(\mu,z)=\tilde{\eta}_{2m+2}(\mu,z)
\allowbreak = 0$ and $m=\infty$.
\end{theorem}

\section*{Acknowledgments}
Financial support from Ministerio de Ciencia e Innovaci\'on, Spain, project PGC2018-098279-B-I00 (MCIU/AEI/FEDER, UE) is acknowledged.

\newpage

\bibliographystyle{siamplain}
\bibliography{biblio}

\begin{thebibliography}{10}

\bibitem{Boyd:1986:UAS}
{\sc W.~G.~C. Boyd and T.~M. Dunster}, {\em Uniform asymptotic solutions of a
  class of second-order linear differential equations having a turning point
  and a regular singularity, with an application to {L}egendre functions}, SIAM
  J. Math. Anal., 17 (1986), pp.~422--450,
  \url{https://doi.org/10.1137/0517033}.

\bibitem{NIST:DLMF}
{\em {\it NIST Digital Library of Mathematical Functions}}.
\newblock http://dlmf.nist.gov/, Release 1.1.1 of 2021-03-15,
  \url{http://dlmf.nist.gov/}.
\newblock F.~W.~J. Olver, A.~B. {Olde Daalhuis}, D.~W. Lozier, B.~I. Schneider,
  R.~F. Boisvert, C.~W. Clark, B.~R. Miller, B.~V. Saunders, H.~S. Cohl, and
  M.~A. McClain, eds.

\bibitem{Dunster:1989:UAW}
{\sc T.~M. Dunster}, {\em Uniform asymptotic expansions for {W}hittaker's
  confluent hypergeometric functions}, SIAM J. Math. Anal., 20 (1989),
  pp.~744--760, \url{https://doi.org/10.1137/0520052}.

\bibitem{Dunster:1994:UAS}
{\sc T.~M. Dunster}, {\em Uniform asymptotic solutions of second-order linear
  differential equations having a simple pole and a coalescing turning point in
  the complex plane}, SIAM J. Math. Anal., 25 (1994), pp.~322--353,
  \url{https://doi.org/10.1137/S0036141092229537}.

\bibitem{Dunster:2003:UAW}
{\sc T.~M. Dunster}, {\em Uniform asymptotic approximations for the {W}hittaker
  functions ${{M}}_{\kappa,i\mu}(z)$ and ${{W}}_{\kappa,i\mu}(z)$}, Anal.
  Appl., 1 (2003), pp.~199--212,
  \url{https://doi.org/10.1142/S0219530503000119}.

\bibitem{Dunster:2020:ASI}
{\sc T.~M. Dunster}, {\em Asymptotic solutions of inhomogeneous differential
  equations having a turning point.}, Stud. Appl. Math., 145 (2020),
  pp.~500--536, \url{https://doi.org/10.1111/sapm.12326}.

\bibitem{Dunster:2020:LGE}
{\sc T.~M. Dunster}, {\em Liouville-{G}reen expansions of exponential form,
  with an application to modified {B}essel functions}, Proc. Roy. Soc.
  Edinburgh Sec. A, 150 (2020), pp.~1289--1311,
  \url{https://doi.org/10.1017/prm.2018.117}.

\bibitem{Dunster:2017:COA}
{\sc T.~M. Dunster, A.~Gil, and J.~Segura}, {\em Computation of asymptotic
  expansions of turning point problems via {C}auchy's integral formula: Bessel
  functions.}, Constr. Approx., 46 (2017), pp.~645--675,
  \url{https://doi.org/10.1007/s00365-017-9372-8}.

\bibitem{Dunster:2020:SEB}
{\sc T.~M. Dunster, A.~Gil, and J.~Segura}, {\em Simplified error bounds for
  turning point expansions}, Anal. Appl.,  (2020),
  \url{https://doi.org/10.1142/S0219530520500104}.

\bibitem{Gaspard:2018:CFB}
{\sc D.~Gaspard}, {\em Connection formulas between coulomb wave functions}, J.
  Math. Phys., 59 (2018), \url{https://doi.org/10.1063/1.5054368}.

\bibitem{Hochstadt:1971:FMP}
{\sc H.~Hochstadt}, {\em The Functions of Mathematical Physics}, John Wiley
  {\&} Sons, Inc., New York-London-Sydney, 1971.

\bibitem{Nestor:1984:UAA}
{\sc J.~J. Nestor}, {\em Uniform Asymptotic Approximations of Solutions of
  Second-order Linear Differential Equations, with a Coalescing Simple Turning
  Point and Simple Pole}, PhD thesis, University of Maryland, College Park, MD,
  1984.

\bibitem{Olver:1975:SOL}
{\sc F.~W.~J. Olver}, {\em Second-order linear differential equations with two
  turning points}, Philos. Trans. R. Soc. A, 278 (1975), pp.~137--174,
  \url{https://doi.org/10.1098/rsta.1975.0023}.

\bibitem{Olver:1980:WFW}
{\sc F.~W.~J. Olver}, {\em Whittaker functions with both parameters large:
  uniform approximations in terms of parabolic cylinder functions}, Proc. Roy.
  Soc. Edinburgh Sect. A, 86 (1980), pp.~213--234,
  \url{https://doi.org/10.1017/S0308210500012130}.

\bibitem{Olver:1997:ASF}
{\sc F.~W.~J. Olver}, {\em Asymptotics and special functions}, AKP Classics, A
  K Peters Ltd., Wellesley, MA, 1997.
\newblock Reprint of the 1974 original [Academic Press, New York].

\end{thebibliography}

\end{document}